\documentclass[12pt,a4paper]{article}
\usepackage{amsmath}
\usepackage{amsfonts}
\usepackage{amsthm}
\usepackage{amscd}
\usepackage{amsbsy}
\usepackage{amsrefs}
\usepackage[latin2]{inputenc}
\usepackage{t1enc}
\usepackage[mathscr]{eucal}
\usepackage{indentfirst}
\usepackage{graphicx}
\usepackage{graphics}
\usepackage{pict2e}
\usepackage{epic}
\numberwithin{equation}{section}
\usepackage[margin=2.9cm]{geometry}
\usepackage{epstopdf}
\usepackage{tikz-cd}
\usepackage{chngcntr}
\usepackage{amssymb}
\usepackage{esint}
\usepackage{hyperref}
\usepackage{mwe}
\usepackage{mathrsfs}
\usepackage{xfrac}
\usepackage{verbatim}
\usepackage{mathpazo}

\theoremstyle{plain}
\newtheorem{Th}{Theorem}[section]
\newtheorem{Lemma}[Th]{Lemma}
\newtheorem{Cor}[Th]{Corollary}
\newtheorem{Prop}[Th]{Proposition}

 \theoremstyle{definition}
\newtheorem{Def}[Th]{Definition}

\newtheorem{Rem}[Th]{Remark}
\newtheorem{?}[Th]{Problem}

\newcommand*\R{\mathbb{R}}

\newcommand*\Om{\Omega}

\newcommand{\Div}{\text{div}_x}

\newcommand{\vectoru}{\mathbf{u}}
\newcommand{\vectorU}{\mathbf{U}}
\newcommand{\vectorv}{\mathbf{v}}
\newcommand{\vectorV}{\mathbf{V}}
\newcommand{\vectorm}{\mathbf{m}}

\newcommand{\dx}{\text{ d}x}
\newcommand{\dt}{\text{ d}t\;}
\newcommand{\dy}{\text{ d}y}

\newcommand{\vectorphi}{\pmb{\varphi}}

\newcommand{\Nu}{\mathcal{V}}

\newcommand{\Ov}[1]{\overline{#1}}
\newcommand{\vc}[1]{{\bf #1}}

\begin{document}

\title{Limit of a consistent approximation to the complete compressible Euler System}

\author{Nilasis Chaudhuri
}
\maketitle

\centerline{Technische Universit{\"a}t, Berlin}

\centerline{Institute f{\"u}r Mathematik, Stra\ss e des 17. Juni 136, D -- 10623 Berlin, Germany.}

\begin{abstract}
	The goal of the present paper is to prove that if a weak limit of a consistent approximation scheme of compressible complete Euler system in the full space $ \R^d,\; d=2,3 $ is a weak solution of the system then eventually the approximate solutions converge strongly in suitable norms locally under a minimal assumption on the initial data of the approximate solutions. The class of consistent approximate solutions is quite general including the vanishing viscosity and heat conductivity limit. In particular,	
	they may not satisfy the \textit{minimal principle for entropy}. 
\end{abstract}


\section{Introduction}
We consider the complete  Euler system in the
physical space $ \R^d$ with $d=2,3$, where the word \emph{complete} means that the system follows the fundamental laws of thermodynamics. The complete Euler system describes the time evolution of the density $ \varrho=\varrho(t,x) $, the momentum $ \vectorm=\vectorm(t,x) $ and the energy $ \mathrm{e}=\mathrm{e}(t,x)$ of a compressible inviscid fluid
in the space time cylinder $ Q_{T}=(0,T)\times \R^d $:
\begin{itemize}
	\item \textbf{Conservation of Mass: }
	\begin{align}
	\partial_t \varrho + \text{div}_x \vectorm&=0. \label{cee:cont}
	\end{align}
	\item \textbf{Conservation of Momentum:}
	\begin{align}
	\partial_t \mathbf{m} + \Div \bigg(\frac{ \mathbf{m} \otimes \mathbf{m}}{\varrho}\bigg)+\nabla_x p&=0. \label{cee:mom}
	\end{align}	
	\item  \textbf{Conservation of energy:}
	\begin{align}
	\partial_{t} \mathrm{e} + \Div\big[(\mathrm{e}+p)\frac{\vectorm}{\varrho} \big]&=0. \label{cee:toen}
	\end{align}
	\end{itemize}
	Here, $ p $ is the pressure related to $ \varrho, \vectorm, \mathrm{e} $ through some suitable equation of {state}.
	\begin{Rem}
		The total energy $ \mathrm{e} $ of the fluid
		$$ \mathrm{e}= \frac{1}{2} \frac{\vert \vectorm \vert^2}{\varrho} + \varrho e,$$
		consists of the kinetic energy $ \frac{1}{2} \frac{\vert \vectorm \vert^2}{\varrho} $ and {the} internal energy $ \varrho e $. 
	\end{Rem}
\begin{itemize}
	\item \textbf{Thermal equation of state:} We introduce the absolute temperature $ \vartheta $. The equation of state is given by Boyle-Mariotte law, i.e. 
	\begin{align}\label{cee:eos}
	e = c_v \vartheta,\; c_v=\frac{1}{\gamma-1},\text{ where } \gamma>1 \text{ is the adiabatic constant.}
	\end{align}
	 The relation between pressure $ p $ and absolute temperature $ \vartheta $ reads as
	 \begin{align*}
	 p=\varrho \vartheta.
	 \end{align*}
\end{itemize}
\begin{Rem}
	As a simple consequence of the previous discussion we have \[ (\gamma-1)\varrho e=p.  \]
\end{Rem}
The second law of thermodynamics is enforced through the entropy balance equation.	
\begin{itemize}	
	\item \textbf{Entropy equation:}
	\begin{align} \label{ent}
	\partial_{t}(\varrho s)+ {\Div(s {\bf m})}=0,
	\end{align}   
	where the entropy is $ s $.
	For smooth solutions, the entropy equations \eqref{ent} can be derived directly from the existing field equations.
	The entropy in terms of the standard variables takes the form:
	\begin{align*}
	s(\varrho,\vartheta)= \log(\vartheta^{c_v})-\log(\varrho).
	\end{align*}
\end{itemize}

\begin{Rem}
	{Now with the introduction of {the} total entropy $S$ by $S=\varrho s$ we rephrase \eqref{ent} as}
	\begin{align*}
	\partial_{t} S + \Div \bigg(S\frac{\vectorm}{\varrho} \bigg)=0.
	\end{align*} 
\end{Rem}
The total entropy helps us to rewrite the pressure $p$ and $ e $ in terms of $ \varrho $ and $ S $ as
\begin{align*}
p=p(\varrho,S)=\varrho^\gamma \exp \bigg(\frac{S}{c_v \varrho}\bigg),\; e=e(\varrho,S)=\frac{1}{\gamma-1}\varrho^{\gamma-1}\exp \bigg(\frac{S}{c_v \varrho}\bigg).
\end{align*}
The advantage of {the} above way of writing is that $ (\varrho,S)\mapsto \varrho^\gamma \exp \bigg(\frac{S}{c_v \varrho}\bigg) $ is a strictly convex function in the domain of positivity, meaning at points, where it is finite and positive. A detailed proof can be found in Breit, Feireisl and Hofmanov\'a \cite{BeFH2019}. Let us complete the formulation of the complete Euler system by imposing the initial and far-field conditions:
\begin{itemize}
	\item \textbf{Initial data:} The initial state of the fluid is given through the conditions
	\begin{align}\label{initial condition}
	{\varrho(0,\cdot)=\varrho_{ 0},\; \vectorm(0,\cdot)= \vectorm_{0},\; S(0,\cdot)=S_0}.
	\end{align}
	\item \textbf{Far field condition:} We introduce the \emph{far field condition} as,
	\begin{align}\label{far_field}
	\varrho \rightarrow {\varrho}_\infty, \; \vectorm \rightarrow \vectorm_\infty,\; S \rightarrow S_\infty \text{ as } \vert x \vert \rightarrow \infty,
	\end{align}
	with $ \varrho_\infty>0,\ \vectorm_\infty \in \R^d$ and $ S_\infty \in \R $.
\end{itemize}\par

	There are many results concerning the mathematical theory of the complete Euler system. It is known that the initial value problem is well posed \emph{locally} in time in the class of smooth solutions, see e.g. the monograph by Majda\cite{M1984} or the recent monograph by {Benzoni--Gavage and Serre} \cite{BS2007}. In Smoller\cite{S1983}, it has been observed that a smooth solution develops singularity {in a finite time}. Thus it is adequate to consider a more general class of {weak (distributional) solutions to study the global in time behavior}. {However, uniqueness may be lost in a larger class of solutions.}\par 

Since our interest is {in} weak or dissipative solutions of the system, we {relax the entropy balance to inequality},
\begin{align}\label{cee:ent_in}
\partial_{t} S + \Div \bigg(S\frac{\vectorm}{\varrho} \bigg)\geq 0,
\end{align}
that is a physically relevant admissibility criteria for weak {solutions}. The adaptation of the method of convex integration in the context of incompressible fluids by De Lellis and Sz\'ekelyhidi \cite{DS2010} leads to ill-posedness of several problems in fluid mechanics also in the class of compressible barotropic fluids, see Chiodaroli and Kreml \cite{CK2014}, Chiodaroli, De Lellis and Kreml \cite{CDK2015} and Chiodaroli et al. \cite{CKMS2018}. The results by Chiodaroli, Feireisl and Kreml\cite{CFK2015} indicate that initial-boundary value problem {for the complete Euler system} admits infinitely many weak solutions on a given time interval $ (0,T) $ for a large class of initial data. In \cite{FKKM2020}, Feireisl et al. show that complete Euler system is ill-posed and these solutions satisfy the entropy inequality \eqref{cee:ent_in}. Chiodaroli, Feireisl and Flandoli in \cite{CFF2019} obtain the similar result for the complete Euler system driven by multiplicative white noise.{ Most of these results, based on the application of the method of convex integration, are non--constructive and use the fact that the constraints imposed by the Euler system on the class of weak solutions allow for oscillations. 
It is therefore of interest to see if solutions of the Euler system can be obtained as a \emph{weak} limit of a suitable approximate sequence. It is our goal to show that it is in fact \emph{not} the case, at least in the geometry of the full space $R^d$.}	\par
In the particular case of constant entropy, the complete Euler system reduces to its isentropic 
(or in a more general setting barotropic) version, where the pressure depends solely on the density.
Compressible barotropic Euler system is expected to describe the vanishing viscosity limit of the compressible barotropic Navier-Stokes system. If compressible barotropic Euler system admits a smooth solution, the unconditional convergence of vanishing viscosity limit has been established by Sueur\cite{Su2014}. Very recently Basari\'c \cite{Bd2019} identified the vanishing viscosity limit of the Navier-Stokes system with a measure valued solution of the barotropic Euler system for the unbounded domains. In \cite{FH2019}, Feireisl and Hofmanov\'a have established that in {the whole space} the vanishing viscosity limit of the barotropic system either converges strongly or its weak limit is not a weak solution for the corresponding barotropic Euler system. \par 
In this article we are interested in the complete Euler system. Feireisl in \cite{F2016} {showed that 
vanishing viscosity limit of the Navier-Stokes-Fourier system in the class of general weak solutions yields the complete Euler system}, provided the later admits smooth solution in bounded domain. Wang and Zhu \cite{WZ2018} establish a similar result in bounded domain with no-slip boundary condition. \par 
Approximate solutions can be viewed as some numerical approximation of {the complete} Euler system. Here we consider a more general class of approximate solutions, namely \emph{consistent approximate solutions}, drawing inspiration from Diperna and Majda \cite{DM1988}. Another example of such approximate problem may be derived from two models, \textbf{Navier-Stokes-Fourier system} and \textbf{Brenner's Model}. A discussion about these models have been presented in B\v rezina and Feireisl \cite{BF2018a}.   \par 
{The consistent approximations typically generate the so--called measure--valued solutions.} For the complete Euler system existence of measure valued solutions has been proved by B\v rezina and Feireisl \cite{BF2018b}, \cite{BF2018a} with the help of Young measures. Later in \cite{BeFH2019}, Breit, Feireisl and Hofmanov\'a  define dissipative solutions for the same system, by modifying the measure-valued solutions suitably. \par

Our main goal is to show that in $ \R^d  $ with $ d=2,3 $, if {approximate solutions} converge weakly to a weak solution of complete Euler system then the convergence will be point-wise almost everywhere. In certain cases we can further establish that the convergence is strong too. Some approximate solutions obtained from the Brenner's model satisfy the \textit{minimal principle for entropy} i.e. if the initial entropy $ s_n(0,\cdot )\geq s_o $ in $ \R^d $ for some constant $ s_0 $, then $$ s_n(t,x)\geq s_0 $$ for a.e.  $ (t,x) \in (0,T)\times \R^d $. Meanwhile this principle is unavailable for approximate solutions obtained from Navier-Stokes-Fourier system. In this paper we consider both type of approximate solutions. As we shall see, the lack of the entropy minimum principle will considerably weaken the available uniform bounds on the approximate sequence. Still we are able to establish strong 
a.e. convergence.
Another important {feature} of our result is that we only assume the {initial energy} is bounded  and the initial data for density converges weakly. {Indeed Feireisl and Hofmanov\'a \cite{FH2019} observed} that if the initial energy converges strongly then similar result can be obtained.\par 
Our plan for the paper is:
	\begin{enumerate}
		\item In Section 2, we recall the definition of weak solutions of the complete Euler system.
		\item In Section 3, we state the approximate problems and main theorems.
		\item In Section 4, few important results have been stated and proved.
		\item In Section 5, we provide the proof of the theorem when approximate solutions satisfy entropy inequality only.
		\item In Section 6, we deal with the renormalized entropy inequality and prove the desired result.
	\end{enumerate}
\section{Preliminaries}
We introduce few standard notations. 
\subsection{Notation}
The space $ C_0(\R^d) $ {is} the closure
under the supremum norm of compactly supported, continuous functions on $ \R^d $ ,
that is the set of continuous functions on $ \R^d $ vanishing at infinity. By $ \mathcal{M}(\R^d) $ we denote the dual space of $ C_0(\R^d) $ consisting of signed Radon measures with finite mass equipped with the dual norm of total variation.\par
The symbol $\mathcal{M}^{+}(\R^d)$ denotes the cone of non-negative Radon measures on $\R^d$ and $\mathcal{P}(\R^d)$ indicates the space of probability measures, i.e. for $\nu \in \mathcal{P}(\R^d)\subset {\mathcal{M}^+} (\R^d)$ we have {
$\nu[\R^d]=1$}. The symbol $\mathcal{M}(\R^d;\R^d)$ means the space of vector valued  finite signed Radon measures and $ \mathcal{M}^{+}(\R^d;\R^{d\times d}_{\text{sym}})  $ denotes the space of symmetric positive semidefinite matrix valued finite signed Radon measures, {meaning $\nu : (\xi \otimes \xi) \in \mathcal{M}^+(\R^d)$ for any $\xi \in \R^d$.}
\par
For $ T>0 $, we denote the space of essentially bounded weak(*) measurable functions from $ (0,T) $ to $ \mathcal{M}(\R^d)  $ by $ L^{\infty}_{\text{weak-(*)}}(0,T;\mathcal{M}(\R^d)) $. Since $ C_0(\R^d) $ is separable Banach space, we have $L^{\infty}_{\text{weak-(*)}}(0,T;\mathcal{M}(\R^d))$ is the dual of $ L^1(0,T; C_0(\R^d)) $. We also observe that $L^{\infty}_{\text{weak-(*)}}(0,T;L^2+\mathcal{M}(\R^d))$ is the dual of $ L^1(0,T; L^2 \cap C_0(\R^d))$. \par
We have introduced the total energy $ \mathrm{e} $ in Section 1. 
{For problems on the full space $ \R^d $ with far field conditions, it is convenient to 
consider a suitable form of relative energy.}
\begin{itemize}
	\item We {denote},  \[\mathrm{e}_{\text{kin}}=\frac{1}{2} \frac{\vert \vectorm \vert^2}{\varrho}\; \text{ and }  \mathrm{e}_{\text{int}}= \frac{1}{\gamma-1}\varrho^{\gamma}\exp \bigg(\frac{S}{c_v \varrho}\bigg)\] and 
	\[ \mathrm{e}(\varrho,\vectorm,S)=\mathrm{e}_{\text{int}}(\varrho,S)+\mathrm{e}_{\text{kin}}(\varrho,\vectorm).\]
	\item Let $ (\varrho_\infty,\vectorm_{\infty}, S_\infty) \in \mathbb{R}\times \mathbb{R}^d \times \mathbb{R} $ such that $ \varrho_\infty>0$.	We define the relative energy with respect to $ (\varrho_\infty, \vectorm_\infty, S_\infty) $as, \[ \mathrm{e}(\varrho,\vectorm,S|\varrho_\infty,\vectorm_\infty,S_\infty)=  \mathrm{e}_{\text{int}}(\varrho,S|\varrho_{\infty},S_\infty)+\mathrm{e}_{\text{kin}}(\varrho,\vectorm| \varrho_\infty,\vectorm_{\infty}), \]
	with 
	\begin{align*}
	 \mathrm{e}_{\text{int}}(\varrho,S|\varrho_{\infty},S_\infty)=  \mathrm{e}_{\text{int}}(\varrho,S)&-\frac{\partial \mathrm{e}_{\text{int}}}{\partial \varrho}(\varrho_\infty, S_\infty) (\varrho-\varrho_\infty)\\ &-\frac{\partial \mathrm{e}_{\text{int}}}{\partial S}(\varrho_\infty, S_\infty) (S-S_\infty) -\mathrm{e}_{\text{int}} (\varrho_\infty,S_\infty)
	\end{align*}
	 and
	\begin{align*}
	 \mathrm{e}_{\text{kin}}(\varrho,\vectorm|\varrho_\infty,\vectorm_{\infty})=&\mathrm{e}_{\text{kin}}(\varrho,\vectorm)-\frac{\partial \mathrm{e}_{\text{kin}}}{\partial \varrho}(\varrho_\infty, \vectorm_\infty) (\varrho-\varrho_\infty)\\ &-\frac{\partial \mathrm{e}_{\text{kin}}}{\partial \vectorm}(\varrho_\infty, \vectorm_\infty) \cdot (\vectorm-\vectorm_\infty) -\mathrm{e}_{\text{kin}} (\varrho_\infty,\vectorm_\infty).
	\end{align*}
	{Introducing the velocity fields $\vectoru$, $\vectoru_\infty$ as $ \vectorm= \varrho \vectoru $  and $ \vectorm_\infty= \varrho_\infty \vectoru_\infty  $, respectively} we observe
	\begin{align*}
	\mathrm{e}_{\text{kin}}(\varrho,\vectoru|\varrho_\infty,\vectoru_{\infty})=\frac{1}{2}\varrho \vert  \vectoru-\vectoru_\infty \vert^2.
	\end{align*}
	\item  In a more precise notation we write
	\begin{align*}
	&\mathrm{e}(\varrho,\vectorm,S|\varrho_\infty,\vectorm_\infty,S_\infty)\\
	& =\mathrm{e}(\varrho,\vectorm,S)-\partial \mathrm{e}(\varrho_\infty,\vectorm_\infty,S_\infty) \cdot [(\varrho,\vectorm,S)-(\varrho_\infty,\vectorm_\infty,S_\infty)]\\
	&\quad  -\mathrm{e}(\varrho_\infty,\vectorm_\infty,S_\infty).
	\end{align*}
\end{itemize}
\par 
We {introduce} the following \textbf{energy extension} in $ \R^{d+2} $ :
\begin{align}\label{energy extension}
\begin{split}
(\varrho,\vectorm, S)\mapsto \mathrm{e}(\varrho,\vectorm,S)\equiv  \begin{cases} 
&\frac{1}{2}\frac{\vert \vectorm \vert^2}{\varrho}+ c_v \varrho^{\gamma} \exp{\bigg(\frac{S}{c_v \varrho}\bigg)},   \text{ if } \varrho > 0,\\
&0,  \text{ if } \varrho = \vectorm  =0,\; S \leq  0 ,\\
&\infty,  \text{ otherwise }
\end{cases}
\end{split}
\end{align}

The above function is a convex lower semi-continuous on $\R^{d+2}$ and strictly convex on its domain of positivity.
\par  
Throughout our discussion we use $ C $ as a positive generic constant that is independent of $ n $ unless specified.  
\subsection{Definition of the weak solution for complete Euler system}
\begin{Def}
	Let $ (\varrho_\infty,\vectorm_{\infty}, S_\infty) \in \mathbb{R}\times \mathbb{R}^d \times \mathbb{R} $ such that $ \varrho_\infty>0$. The triplet $ (\varrho,\vectorm,S) $ is called a \textit{weak solution} of the complete Euler system with initial data $ (\varrho_{ 0},\vectorm_{0},S_0) $,
	if the following system of identities is satisfied:
	\begin{itemize}
		\item \textbf{Measurability:} The variables $ \varrho=\varrho(t,x),\; \vectorm= \vectorm(t,x)\; S=S(t,x) $ are measurable function in $ (0,T)\times \R^d $, \ {$\varrho \geq 0$,}
			\item \textbf{Continuity equation:}
		\begin{align}
		\begin{split}
		\int_0^T \int_{\R^d} \big[\varrho \partial_{t} \phi + \vectorm \cdot \nabla_{x} \phi \big] \dx \dt =-\int_{ \R^d} \varrho_{ 0} \phi(0,\cdot)\dx,
		\end{split}
		\end{align}
		for any $ \phi \in C_c^1([0,T)\times \R^d) $.
		\item \textbf{Momentum equation:}
		\begin{align}
		\begin{split} 
		&\int_0^T \int_{\R^d} \bigg[\vectorm \cdot \partial_{t} \vectorphi + {\mathbb{1}_{\{ \varrho > 0\}} }\frac{\vectorm \otimes \vectorm}{\varrho} : \nabla_{x} \vectorphi + {\mathbb{1}_{\{ \varrho > 0\}} } p(\varrho,S) \Div \vectorphi \bigg] \dx \dt \\
		&= -\int_{ \R^d} \vectorm_{0}\cdot \vectorphi(0.\cdot) \dx,
		\end{split}
		\end{align}
		for any $ \vectorphi \in C_c^1([0,T)\times \R^d;\R^d) $.
		\item \textbf{Relative energy inequality:}
		{The satisfaction of the far field conditions is enforced through the relative energy inequality} in the following form : 
		\begin{align}
		\begin{split}
			\bigg[ \int_{\R^d} \mathrm{e}(\varrho,\vectorm,S|\varrho_\infty,\vectorm_\infty,S_\infty)\; (\tau,\cdot) \dx \bigg]^{\tau=t}_{\tau=0}\leq 0,
		\end{split}
		\end{align}
		for a.e. $ t\in (0,T) $.
		\item \textbf{Entropy inequality:}
		\begin{align}
		\int_0^T \int_{\R^d} \bigg[S\;  \partial_{t} \phi + \mathbb{1}_{\{\varrho>0\}}\frac{S }{\varrho} \vectorm  \cdot \nabla_{x} \phi \bigg] \dx \dt \leq 0,
		\end{align}
		for any $ \phi \in C_c^1((0,T)\times \R^d) $ with $\phi \geq 0$.
	\end{itemize}

\end{Def} 
Note that the above definition of admissible weak solution is considerably weaker than the standard weak formulation that contains also the energy balance \eqref{cee:toen}. The present setting is more in the spirit of more general measure--valued solutions introduced in B\v rezina, Feireisl \cite{BF2018b}. As a matter of fact, considering weaker concept of generalized solutions makes our results stronger as the standard weak solutions are covered.
\section{Approximate problem and Main theorems}
As we have mentioned in the introduction our main results are related to the approximate problems of the complete Euler system. 
Let $ (\varrho_\infty,\vectorm_{\infty}, S_\infty) \in \mathbb{R}\times \mathbb{R}^d \times \mathbb{R} $ such that $ \varrho_\infty>0$.
 \subsection{Approximate problems of complete Euler system }
We say $ (\varrho_{ n},\vectorm_{ n},S_n=\varrho_{ n}s_n) $ is a family of admissible consistent approximate solutions for the complete Euler system in $ (0,T)\times\R^d $ with initial data $ (\varrho_{ 0,n},\vectorm_{0,n}, S_{0,n}=\varrho_{ 0,n} s_{0,n}) $ if the following holds:
\begin{itemize}
	\item {$\varrho_{ n}\geq0$} and any $ \phi \in C_c^1([0,T)\times \R^d) $ we have,
	\begin{align}\label{app_cont}
	-\int_{\R^d} \varrho_{ 0,n} \phi(0,\cdot) \dx=\int_0^T \int_{\R^d} \big[\varrho_n \partial_{t} \phi + \vectorm_n \cdot \nabla_{x} \phi \big] \dx \dt + \int_0^T \mathfrak{E}_{1,n}[\phi] \dt ;
	\end{align}
	\item For any $ \vectorphi \in C_c^1([0,T)\times \R^d;\R^d) $, we {have}
	\begin{align}\label{app_mom}
	\begin{split}
	&\hspace{-2mm}-\int_{\R^d}  \vectorm_{0,n} \vectorphi(0,\cdot) \dx\\
	&\hspace{-2mm}=\int_0^T  \int_{\R^d} \bigg[ \vectorm_n \cdot \partial_{t} \vectorphi + \mathbb{1}_{\{\varrho_{ n}>0\}}\frac{\vectorm_n \otimes \vectorm_n}{\varrho_n}\colon \nabla_{x} \vectorphi + {\mathbb{1}_{\{\varrho_{ n}>0\}}} p(\varrho_n,S_n) \Div \vectorphi \bigg] \hspace{-2mm}\dx \hspace{-1mm}\dt\\
	&\quad +\int_0^T \mathfrak{E}_{2,n}[\vectorphi] \dt;
	\end{split}
	\end{align}
	\item \text{ For a.e. }$ 0\leq \tau \leq T$, we have 
	\begin{align}\label{app_engy}
	\begin{split}
	&\int_{ \R^d} \mathrm{e}(\varrho_{n},\vectorm_{n},S_n|\varrho_\infty,\vectorm_\infty,S_\infty)(\tau) \dx\\
	& \leq\int_{ \R^d} \mathrm{e}(\varrho_{0,n},\vectorm_{0,n},S_{0,n}|\varrho_\infty,\vectorm_\infty,S_\infty) \dx+\mathfrak{E}_{3,n};
	\end{split}
	\end{align}
	\item  	For any $ \psi \in C_c^1([0,T)\times \R^d) $ with $\psi \geq 0$, we have
	\begin{align}\label{app_entropy}
	\begin{split}
	&\int_0^T \int_{\R^d} \bigg[S_n\;  \partial_{t} \psi + \mathbb{1}_{\{\varrho_{ n}>0\}} 
	{\frac{S_n}{\varrho_{ n}}\vectorm_n  \cdot \nabla_{x} \psi} \bigg] \dx \dt \\
	&\leq  -\int_{\R^d} \varrho_{ 0,n} s_{0,n} \psi(0,\cdot) \dx+  \int_0^T\mathfrak{E}_{4,n}[\psi]\dt ;
	\end{split}
	\end{align}
	\item 
	Here, the terms $\mathfrak{E}_{1,n}[\phi],\; \mathfrak{E}_{2,n}[\vectorphi],\; \mathfrak{E}_{3,n} \text{ and }\mathfrak{E}_{4,n}[\psi]$ represent \emph{consistency error}, i.e.,
	\[ \mathfrak{E}_{3,n},\; \mathfrak{E}_{4,n}[\psi] \geq 0 \]
	and 
	\begin{align}\label{app_const_er}
	\mathfrak{E}_{1,n}[\phi]\rightarrow 0,\; \mathfrak{E}_{2,n}[\vectorphi]\rightarrow 0,\;\mathfrak{E}_{3,n}\rightarrow 0 \text{ and }\mathfrak{E}_{4,n}[\psi]\rightarrow 0 \text{ as } n\rightarrow \infty ,
	\end{align}
	for fixed {$\phi, \; \vectorphi\text{ and } \psi(\geq 0)$}.
\end{itemize}

	Instead of \eqref{app_entropy}, a renormalized version of entropy inequality for approximate problem can be considered:
	\begin{align}\label{app_entropy_renorm}
	\begin{split}
	\int_0^T \int_{\R^d} \bigg[\varrho_{ n}\chi(s_n)\;  \partial_{t} \psi + {\chi(s_n) }\vectorm_n  \cdot \nabla_{x} \psi \bigg] \dx \dt  \leq -\int_{\R^d} \varrho_{ 0,n} \chi(s_{0,n}) \psi(0,\cdot) \dx,
	\end{split}
	\end{align}
	for any $ \psi \in C_c^1([0,T)\times \R^d) $ with $\psi \geq 0$ and any $\chi$ and $ \bar{\chi} \in \R^{+}$ with
	\[\chi:\R \rightarrow \R \text{ a non--decreasing concave function, }\chi(s)\leq \bar{\chi} \text{ for all }s\in \R. \]
	\begin{Rem}
		Clearly one can recover the inequality \eqref{app_entropy} without error from the \eqref{app_entropy_renorm}. Further, consideration of renormalized entropy inequality \eqref{app_entropy_renorm} leads us to conclude that the entropy is transported along streamlines, see B\v rezina and Feireisl \cite{BF2018b}. We rephrase it by saying the \textit{minimal principle for entropy } holds, i.e.
		\begin{align}\label{ent_min_app2}
		\text{for } s_0\in \R,  \text{ if }s_n(0,\cdot) \geq s_0 \text{ then }	s_n(\tau,\cdot)\geq s_0 \text{ in } \R^d \text{ for a.e. } 0\leq \tau\leq T.
		\end{align} 
	\end{Rem}
	
\begin{Rem}
	It has been shown in \cite{BF2018b} that approximate solutions comes from Navier-Stokes-Fourier system may not satisfy renormalized version of entropy balance \eqref{app_entropy_renorm}, they only satisfy \eqref{app_entropy}. While solutions coming from Brenner's model satisfy \eqref{app_entropy_renorm}.
\end{Rem}
From now on we refer as follows:
\begin{itemize}
	\item \textbf{First approximation problem  }: Approximate solutions satisfy \eqref{app_cont}-\eqref{app_const_er};
	\item  \textbf{Second approximation problem }: Approximate solutions satisfy \eqref{app_cont}-\eqref{app_engy}  \eqref{app_const_er}, and \eqref{app_entropy_renorm}.
\end{itemize}
\subsection{Hypothesis on the initial data}
We assume that initial density is non-negative and initial relative energy is uniformly bounded, i.e.
\begin{align}\label{app_ic_1}
\begin{split}
&{\varrho_{ 0,n} \geq 0 \ \mbox{and}} {\int_{ \R^d} \mathrm{e}(\varrho_{0,n},\vectorm_{0,n},S_{0,n}|\varrho_\infty,\vectorm_\infty,S_\infty) \dx \leq E_0}\\
\end{split}
\end{align}
with $ E_0$ is independent of $ n $.
{As the relative energy is strictly convex in its domain, we obtain}  
\begin{align}\label{app_ic1_1.2}
&\varrho_{ 0,n}-\varrho_\infty \in L^2+L^1(\R^d) {\text{ and } \varrho_{ 0,n} \rightharpoonup \varrho_0 \text{ in } 
\mathcal{M}^+_{\text{loc}}(\R^d )\text{ as  } n \rightarrow \infty}
\end{align}
{passing to a subsequence as the case may be.}\\
This is enough for the \textbf{first approximation problem} but the \textbf{second approximation problem} needs some additional assumption that the initial entropy is bounded below i.e. for some $ s_0\in \R $ we have 
\begin{align}\label{app_ic_2}
{s_{0,n}} \geq s_0 \text{ in } \R^d, \text{ for all } n\in \mathbb{N} .
\end{align}  
\subsection{Main theorem}
Before stating our main results, we observe that hypothesis \eqref{app_ic_1} shared by both approximate problems yields uniform bounds
\[
\varrho_n - \varrho_\infty, \ S_n - S_\infty \in L^\infty(0,T; L^1 + L^2(\R^d)), \ 
{\bf m}_n - {\bf m}_\infty \in L^\infty(0,T; L^1 + L^2(\R^d)).
\]
{In particular, passing to a subsequence if necessary, we may assume that the sequence 
$(\varrho_n, {\bf m}_n, S_n)$ generates a Young measure $\{ \mathcal{V}_{t,x} \}_{t \in (0,T) \times \R^d}$, as described in Ball \cite{Ball1989}.
We denote}
\[
(\varrho(t,x),\vectorm(t,x),S(t,x))=(\langle \Nu_{t,x}; \tilde{\varrho}\rangle, \langle \Nu_{t,x};\tilde{\vectorm} \rangle,\langle \Nu_{t,x}; \tilde{S} \rangle).
\]
We also observe that
\[
(\varrho, {\bf m}, S) \in L^{\infty}_{\text{weak(*)}}(0,T ; L^1_{{\rm loc}}(\R^d)).
\]
As we have noticed that the fundamental difference of two approximate problem is the minimal condition for entropy.  Here we will state the main theorems.
\begin{Th}[\textbf{First approximation problem}]\label{th1}
	Let $d=2,3$ and $ \gamma>1 $. Let $ (\varrho_n, \vectorm_n, S_n= \varrho_{ n} s_n ) $ be a sequence of \emph{admissible solutions} of the consistent approximation with uniformly bounded initial energy as in \eqref{app_ic_1} and {the initial densities satisfying \eqref{app_ic1_1.2}. 
Suppose that the barycenter $(\varrho, {\bf m}, S)$ of the Young measure generated by the sequence
$(\varrho_n, {\bf m}_n, S_n)$ is an admissible weak solution of the complete Euler system satisfying} 
	\begin{align}\label{assump_ent_app1}
	\varrho(0, x) = \varrho_0(x),\ 
	S(t,x)=0 \text{ whenever } \varrho(t,x)=0 \text{ for  a.e. } (t,x) \in(0,T)\times \R^d.
	\end{align}
	Then 
\begin{align*}
	&\varrho_{ n} \rightarrow \varrho \text{ in } L^{q}(0,T;L^1_{\rm loc}(\R^d)),\;\vectorm_n  \rightarrow \vectorm \text{ in }L^{q}(0,T;L^{1}_{\rm loc}(\R^d;\R^d))	,\\
	& S_n\rightarrow S\text{ for a.e. } (t,x)\in (0,T)\times  \R^d ,
	\end{align*}
	for any $1 \leq q < \infty$.
Moreover, 
passing to a subsequence as the case may be, we have
	\begin{align}
	\begin{split}
	&\varrho_{ n}\rightarrow \varrho ,\;  \vectorm_n \rightarrow \vectorm \text{ and}\; S_n \rightarrow S \text{ for a.e.} (t,x)\in (0,T)\times \R^d .\\
	\end{split}
	\end{align}

\end{Th}

\begin{Th}[\textbf{Second approximation problem}]\label{th2}
		Let $d=2,3$ and $ \gamma>1 $. Let  \\$ (\varrho_n, \vectorm_n, S_n= \varrho_{ n} s_n ) $ be a sequence of \emph{admissible solutions} of the consistent approximation with initial energy satisfying \eqref{app_ic_1} and the initial entropy satisfying \eqref{app_ic_2}. Suppose, 
	\begin{align}
	\begin{split}
	&\varrho_{ n}\rightarrow \varrho \text{ in } \mathcal{D}^{\prime}((0,T)\times\R^d),\;  \vectorm_n \rightarrow \vectorm \text{ in }\mathcal{D}^{\prime}((0,T)\times\R^d;\R^d),\; \\&S_n \rightarrow S \text{ in }\mathcal{D}^{\prime}((0,T)\times\R^d),\\
	\end{split}
	\end{align}	
	where $ (\varrho,\vectorm, S) $ is a weak solution of the complete Euler system. \\
	Then
	\[ \mathrm{e}(\varrho_{ n},\vectorm_{n},S_n|\varrho_\infty,\vectorm_\infty,S_\infty)\rightarrow \mathrm{e}(\varrho,\vectorm,S|\varrho_\infty,\vectorm_\infty,S_\infty) \text{ in } {L^{q}(0,T;L^1_{\rm loc}(\R^d))} \]
	as $ n\rightarrow \infty  $ for any $ 1\leq q <\infty $. {Moreover,}
	\begin{align*}
	&\varrho_{ n} \rightarrow \varrho \text{ in } L^{q}(0,T;L^\gamma_{\rm loc}(\R^d)), \vectorm_n  \rightarrow \vectorm \text{ in }L^{q}(0,T;L^{\frac{2\gamma}{\gamma+1}}_{\text{loc}}(\R^d;\R^d)) \\
	& S_n \rightarrow S \text{ in }L^{q}(0,T;L^\gamma_{\rm loc}(\R^d)),
	\end{align*}
for any $1 \leq q < \infty$.	
\end{Th}
\section{{Essential} results}
As is well known, uniformly bounded sequence in $ L^1(\R^d) $ does not imply, in general, weak convergence of it. Using the fact that $L^1(\R^d)  $ is continuously embedded in space of Radon measures $ \mathcal{M} (\R^d)$ and identification of $ \mathcal{M}(\R^d) $ with the dual of $ C_0(\R^d) $ provides some weak(*) compactness. On the other hand by Chacon's biting limit theorem characterizes that the limit measure concentrate in some subsets of $ \R^d $ with small Lebesgue measure and other than this small sets the limit is a $ L^1  $ function.
\subsection{Concentration defect measure}
In this section we will establish few results.
Let $ \vectorU_n :\R^d \rightarrow \R^m  $ such that 
\begin{align*}
&\vectorU_n=\vectorV_n+ \mathbf{W}_n \text{ with }\Vert \vectorV_n \Vert_{L^2(\R^d;\R^m)} + \Vert \mathbf{W}_n \Vert_{L^1(\R^d;\R^m)} \leq C, 
\end{align*}
$ C $ is independent of $ n $. 
 Fundamental theorem of Young measure as in \cite{Ball1989} ensures the existence of a Young measure $ \nu \in L^{\infty}_{\text{weak-(*)}}(\R^d;\mathcal{M}(\R^d;\R^m))$, generated by $ \{\vectorU_n\}_{n\in \mathbb{N}} $. Further we have $ y \mapsto \left \langle  \nu_{y}, \tilde{\vectorU}\right \rangle \; \text{in } L^1_{\text{loc}}(\R^d) . $ \par
It is well known fact that $ L^{1}(\R^d;\R^m)\subset \mathcal{M}(\R^d;\R^m) $. We can conclude that as $ n \rightarrow \infty $,
\begin{align*}
\begin{split}
&\vectorV_n  \rightarrow \vectorV, \text{ weakly in } L^2{(\R^d;\R^m)},\\
&\mathbf{W}_n \rightarrow \mu_{\mathbf{W}}, \text{ weak(*)ly in } \mathcal{M}(\R^d;\R^m).
\end{split}
\end{align*}
We define $ \overline{\vectorU}= \vectorV+\mu_{\mathbf{W}} $. Clearly $ \overline{\vectorU} \in \mathcal{D}^{\prime}(\R^d;\R^m) $. 
\begin{Def}
	The quantity $ \mathfrak{C}_{\vectorU}=\overline{\vectorU} - \{y \mapsto \langle \nu_y ; \tilde{\vectorU} \rangle \} $ has been termed as \textbf{concentration defect measure.} 
\end{Def}

Here we state a result that gives us the comparison of defect for two different nonlinearities.
\begin{Lemma} \label{LA1}
	
	Suppose $ \mathbf{U}_n: Q(\subset \R^d )\rightarrow \R^m $ and $E: R^m \to [0, \infty]$ is a lower semi-continuous function, 
	\begin{equation} \label{A1}
	E(\vc{U}) \geq |\vc{U}| \ \mbox{as}\ |\vc{U}| \to \infty,
	\end{equation}
	and let $\vc{G}: R^m \to R^n$ be a continuous function such that
	\begin{equation}\label{A2}
	\limsup_{|\vc{U}| \to \infty} | \vc{G}(\vc{U}) | < \liminf_{|\vc{U}| \to \infty} E(\vc{U}).
	\end{equation}
	
	Let $\{ \vc{U}_n \}_{n=1}^\infty$ be a family of measurable functions,
	\begin{equation*} 
	\int_Q E(\vc{U}_n)  \dy \leq 1.
	\end{equation*}
	
	Then 
	\begin{equation} \label{A4}
	\Ov{E(\vc{U})} - \left<\nu_y; E(\widetilde{\vc{U}}) \right> \geq  
	\left| \Ov{ \vc{G} (\vc{U}) } -  \left<\nu_y; \vc{G} (\widetilde{\vc{U}}) \right> \right|.
	\end{equation}
\end{Lemma}
\begin{Rem} \label{AR1}
	
	Here $\Ov{E(\vc{U})} \in \mathcal{M}^+(Q)$ and $\Ov{ \vc{G} (\vc{U}) } \in \mathcal{M}(Q; R^n)$ are the corresponding weak(*) limits and $\nu$ denotes the Young measure generated by $\{ \vc{U}_n \}$. {The inequality \eqref{A4} should be understood as} 
\[
\Ov{E(\vc{U})} - \left<\nu_y; E(\widetilde{\vc{U}}) \right> - 
\left( \Ov{ \vc{G} (\vc{U}) } -  \left<\nu_y; \vc{G} (\widetilde{\vc{U}}) \right> \right) \cdot \xi \geq 0
\]
{for any $\xi \in R^n$, $|\xi| = 1$}.
	
\end{Rem}

\begin{proof}
	
	The result was proved for continuous functions $E$, $\vc{G}$, see Lemma 2.1\cite{FPAW2016}. 
	To extend it to the class of lower semi-continuous functions like $E$, we first observe that there is a sequence of continuous functions 
	$F_n \in C(R^m)$ such that 
	\[
	0 \leq F_n \leq E,\ F_n \nearrow E.
	\]
	In view of \eqref{A2}, there exists $R > 0$ such that 
	\[
	|\vc{G}(\vc{U})| < E(\vc{U}) \ \mbox{whenever}\ |\vc{U}| > R.
	\]
	Consider a function 
	\[
	T: C^\infty (R^m), \ 0 \leq T \leq 1, \ T(\vc{U}) = 0 \ \mbox{for}\ |\vc{U}| \leq R,\ 
	T(\vc{U}) = 1 \ \mbox{for}\ |\vc{U}| \geq R + 1. 
	\]	
	Finally, we construct a sequence 
	\[
	E_n(\vc{U}) = T (\vc{U}) \max\{ |\vc{G}(\vc{U})|; F_n (\vc{U}) \}.
	\]
	We have 
	\[
	0 \leq E_n (\vc{U}) \leq E(\vc{U}),\ E_n (\vc{U}) \geq |\vc{G}(\vc{U})| \ \mbox{for all}\  |\vc{U}| \geq R + 1.
	\]
	Applying Lemma 2.1 in \cite{FPAW2016} we get 
	\[
	\Ov{E_n(\vc{U})} - \left<\nu_y; E_n(\widetilde{\vc{U}}) \right> \geq  
	\left| \Ov{ \vc{G} (\vc{U}) } -  \left<\nu_y; \vc{G} (\widetilde{\vc{U}}) \right> \right|  
	\]
	for any $n$. 
	Thus the proof reduces to showing 
	\[
	\Ov{E_n(\vc{U})} - \left<\nu_y; E_n(\widetilde{\vc{U}}) \right> \leq 
	\Ov{E(\vc{U})} - \left<\nu_y; E(\widetilde{\vc{U}}) \right>,   
	\]
	or, in other words, to showing 
	\begin{equation*} 
	\Ov{H(\vc{U})} - \left<\nu_y; H(\widetilde{\vc{U}}) \right> \geq 0 \ 
	\mbox{whenever}\ H: R^m \to [0, \infty] \ \mbox{is an l.s.c function.}
	\end{equation*}
	Repeating the above arguments, we construct a sequence 
	\[
	0 \leq H_n \leq H \ \mbox{of bounded continuous functions, }\  H_n \nearrow H. 
	\]
	Consequently, 
	\[
	0 \leq \Ov{H(\vc{U}}) - \Ov{H_n(\vc{U}}) = \Ov{H(\vc{U}}) - \left< \nu_y; H_n (\widetilde{\vc{U}}) \right>
	\to \Ov{H(\vc{U}}) - \left< \nu_y; H (\widetilde{\vc{U}}) \right>
	\]
	as $n \to \infty$.
	
\end{proof}
\subsection{Consequences of finiteness of a concentration defect}
Feireisl and Hofmanov\'a in \cite{FH2019} Proposition 4.1 have been proved the following proposition:
\begin{Prop}\label{Def_prop2}
	Let $\mathbb{D} \in \mathcal{M}^{+}(\R^d;\R^{d\times d}_{\text{sym}}) $ satisfy
	\begin{align*}
	\int_{ \R^d} \nabla_{x} \vectorphi : \text{d} \mathbb{D} =0 \text{ for any } \vectorphi \in C_c^1(\R^d;\R^d), 
	\end{align*} 
	Then $\mathbb{D} =0$.	
\end{Prop}  
The key ingredient of the proof is the consideration of the sequence of cut off function $ \{\chi_n\}_{n\in \mathbb{N}} $ such that
\begin{align}\label{Cutoff1}
\begin{split}
&\chi_n \in C_c^{\infty} (\R^d),\; 0\leq \chi_n \leq 1, \; \chi_n(x)=1  \text{ for } \vert x \vert \leq n,\; \chi_n(x)=0 \text{ for } \vert x \vert \geq 2n, \; \\
&\vert \nabla_{x} \chi_n \vert  \leq \frac{2}{n} \text{ uniformly as } n\rightarrow \infty.
\end{split}
\end{align}
That leads us to conclude the next result,
\begin{Cor}\label{Def_prop2:cor1}
	Let $ \mathbb{D} =\{\mathbb{D}_{ij}\}_{i,j=1}^{d} \in L^{\infty}_{\text{weak-(*)}}((0,T;\mathcal{M}(\R^d;\R^{d\times d}))$
	{be such that}
	\begin{align*}
	\int_{0}^{T} \int_{ \R^d} \nabla_{x} \phi : \text{ d}\mathbb{D} \;dt =0 
	\end{align*}
	{for any $\phi \in \mathcal{D}((0,T)\times \R^d;\R^d)$}.\\
	Then, for any {$ \psi \in C_c^{\infty}(0,T;C^1(\R^d;\R^d))$, $\nabla_x \psi \in C_c^\infty (0,T;L^\infty(\R^d; \R^{d \times d})) $}, we have
	\begin{align*}
	\int_{0}^{T} \int_{ \R^d} \nabla_{x} \psi : \text{ d}\mathbb{D}\; dt =0. 
	\end{align*}
\end{Cor}
Here we state a lemma that is quite similar to above proposition. The difference here is instead of matrix valued measure we consider vector valued measure.
\begin{Lemma}\label{Def_lem2}
	Let $ \mathbb{D}=\{\mathbb{D}_{i}\}_{i=1}^{d} \in L^{\infty}_{\text{weak-(*)}}((0,T;\mathcal{M}(\R^d;\R^d))$ 
	{be such that}
	\begin{align*}
	\int_{0}^{T} \int_{ \R^d} \nabla_{x} \phi \cdot \text{ d}\mathbb{D}\; dt =0, \text{ for any }\phi \in \mathcal{D}((0,T)\times \R^d) . 
	\end{align*}
	Then, for {any $ \psi \in C_c^{\infty}(0,T;C^1(\R^d) \cap W^{1, \infty}(\R^d)) $, $\nabla_x \psi \in L^\infty(\R^d; \R^{d}) $
	}, we have
	\begin{align*}
	\int_{0}^{T} \int_{ \R^d} \nabla_{x} \psi \cdot \text{ d}\mathbb{D} \;dt =0. 
	\end{align*}
\end{Lemma}
\subsection{Convergence result }
Here we state two convergence results. 
\begin{Lemma}\label{ym1}
	Let $ \{\vectorv_n\}_{n\in \mathbb{N}} $, $ \vectorv_n :\R^d \rightarrow \R^m  $,  
$\{\vectorv_n\}_{n\in \mathbb{N}}$ bounded in $L^1_{\rm loc}(R^d,; R^m)$,	
	generate a Young measure $ \nu $. Suppose $ \vectorv(y)=\langle \nu_y; \tilde{\vectorv} \rangle $ is the barycenter of the Young measure and $ \nu_y= \delta_{\vectorv (y)} $ for a.e. $y \in R^d$, then $ \vectorv_n \rightarrow \vectorv \text{ in measure}.$
\end{Lemma}
\begin{Lemma}\label{wl1}
	Let $Q \subset R^d$ be a bounded domain, and let $\{ \mathbf{v}_n \}_{n=1}^\infty$ be sequence of vector--valued functions, 
	\[
	\mathbf{v}_n : Q \to R^k , \ \int_Q{ | \mathbf{v}_n| } \leq c \ \mbox{uniformly for}\ n \to \infty,
	\]
	generating a Young measure $\nu_y \in \mathcal{P}[R^k]$, $y \in Q$. Suppose that 
	\begin{equation*} 
	E(\mathbf{v}_n) \to \left< \nu_y ; E(\widetilde{\mathbf{v}}) \right> \ \mbox{weak(*)ly in}\ \mathcal{M}(\overline{Q}),
	\ \left< \nu_y; E(\widetilde{\mathbf{v}}) \right> \in L^1(Q),
	\end{equation*}
	where $E: R^d \to [0, \infty]$ is an l.s.c. function. 
	
	Then 
	\[
	E(\mathbf{v}_n) \to \left< \nu_y ; E(\widetilde{\mathbf{v}}) \right> \ \mbox{weakly in}\ L^1(Q).
	\]
\end{Lemma}
\begin{proof}
	Enough to prove the equi-integrability of $ \{E(\vectorv_n)\}_{n\in\mathbb{N}} $. A detailed proof is in \cite{FLMS_book}. 
\end{proof}
\section{Convergence of approximate solutions from the first approximation problem  }
In this section our main goal is to prove the Theorem \ref{th1}. In the formulation of problem we consider that the approximate solutions satisfy weak form of entropy inequality only. We are unable to establish \emph{the minimal principle for entropy}. Now using \eqref{energy extension} and convexity of relative energy, we have
 \begin{align}\label{rel_eng}
\begin{split}
\mathrm{e}(\varrho,\vectorm,S|\varrho_\infty,\vectorm_\infty,S_\infty)\geq  \begin{cases}
(\varrho-\varrho_\infty)^2 + \vert \vectorm-\vectorm_\infty \vert^2 + ({S}- S_\infty)^2 \\
\quad \text{ if }{\frac{\varrho_\infty}{2} \leq \varrho \leq 2\varrho_\infty}\text{ and } {|S| \leq 2 |S_\infty|},\\
\vert \varrho-\varrho_\infty \vert + \vert \vectorm -\vectorm_\infty \vert + \vert S-S_\infty\vert ,
\\
\quad \text{otherwise.}
\end{cases}
\end{split}
\end{align}

\subsection{Uniform Bounds}
From our assumption on initial data \eqref{app_ic_1} we obtain
\begin{align} \label{eng_bd_1}
\Vert \mathrm{e}(\varrho_n,\vectorm_n,S_n|\varrho_\infty,\vectorm_\infty,S_\infty) \Vert_{L^\infty(0,T;L^1(\R^d))} \leq C.
\end{align}
Hence, uniform relative energy bound \eqref{rel_eng} and \eqref{eng_bd_1} imply
\begin{align}\label{den,mom_1}
\begin{split}
\Vert \varrho_n - \varrho_\infty \Vert_{L^{\infty}(0,T;L^1+L^2(\R^d))} &+
\big\Vert \vectorm_n - \vectorm_\infty \big\Vert_{L^{\infty}(0,T;L^{1}+L^2(\R^d;\R^d))} \\
&+ \Vert S_n - S_\infty \Vert_{L^{\infty}(0,T;L^1+L^2(\R^d))} \leq C.
\end{split}
\end{align}
\subsection{Defect measures for state variables $ \varrho, \vectorm $ and $ S $}
   Let us consider $ Z_n=(\varrho_{ n},\vectorm_n,S_n) $. From \eqref{den,mom_1} we conclude that the sequence $\{Z_n\}_{n\in \mathbb{N}}$ is bounded in $L^\infty(0,T; L^1_{\rm loc}({R^d}; \R^{d+2}))$. Thus using the fundamental theorem of Young measure as in Ball \cite{Ball1989} we ensure the existence of $\Nu  $ generated by $\{ Z_n \}_{n\in\mathbb{N}} $ and 
   $$ \Nu \in {L^{\infty}_{\text{weak-(*)}}((0,T)\times R^d; \mathcal{P}(\R \times \R^d \times \R))}  .$$
   On the other hand, we obtain
   \begin{align*}
   \varrho_{ n}-\varrho_{\infty} \rightarrow \overline{\varrho}-\varrho_\infty \text{ as  }n\rightarrow \infty \text{ weak-(*)ly in }L^{\infty}_{\text{weak-(*)}}(0,T; L^2+ \mathcal{M}(\R^d)) .
   \end{align*}
   We introduce the defect measure 
   \begin{align}\nonumber
   \mathfrak{C}_{\varrho}= \overline{\varrho}-\{(t,x)\mapsto \langle \Nu_{t,x}; \tilde{\varrho}\rangle\} 
   \end{align}
   and obtain, {by virtue of Lemma \ref{LA1},} $ \mathfrak{C}_{\varrho}\in L^{\infty}_{\text{weak-(*)}}(0,T; \mathcal{M}(\R^d))   $.\\
   Similarly for the sequences $\{(\vectorm_{ n}-\vectorm_{\infty}) \}_{n\in \mathbb{N}} $ and $\{ (S_n-S_\infty )\}_{n\in\mathbb{N}}$, we define the corresponding concentration defect measures as:
   \begin{align}\nonumber
   \begin{split}
   &\mathfrak{C}_{\vectorm}= \overline{\vectorm}- \{(t,x)\mapsto \langle \Nu_{t,x}; \tilde{\vectorm}\rangle\}\text{ and }\mathfrak{C}_{S}= \overline{S}- \{(t,x)\mapsto \langle \Nu_{t,x}; \tilde{S}\rangle\}.
   \end{split}
   \end{align}
   Using the fact $ \varrho_{ n} \geq 0 $ we can conclude 
	\begin{align}\nonumber
   \mathfrak{C}_{\varrho}\in L^{\infty}_{\text{weak-(*)}}(0,T; \mathcal{M}^{+}(\R^d)).  
   \end{align}
   We denote the barycenter of the Young measure as $ (\varrho, \vectorm, S) $ i.e.
   \begin{align}\nonumber
   \begin{split}
   &(\varrho(t,x),\vectorm(t,x),S(t,x))\\
   &=(\{(t,x)\mapsto \langle \Nu_{t,x}; \tilde{\varrho}\rangle\},  \{(t,x)\mapsto \langle \Nu_{t,x}; \tilde{\vectorm}\rangle\}, \{(t,x)\mapsto \langle \Nu_{t,x}; \tilde{S}\rangle\}).
   \end{split}
   \end{align}
   \begin{Rem}
   	As pointed out by Ball and Murat in \cite{BallM1989}, this baycenter coincides with the biting limit of the sequence $ \{Z_n \}_{n\in \mathbb{N}} $. 
   \end{Rem}
\subsection{Defect measures from non-linear terms}
    \subsubsection{Relative energy defect}
    We recall $L^{\infty}_{\text{weak-(*)}}(0,T;\mathcal{M}(\R^d))$ is the dual of $ L^1(0,T; C_0(\R^d)) $ and relative energy is uniformly bounded \eqref{eng_bd_1}. Passing to a suitable subsequence we obtain
    $$ \mathrm{e}(\varrho_n,\vectorm_n,S_n|\varrho_\infty,\vectorm_\infty,S_\infty)\rightarrow \overline{\mathrm{e}(\varrho,\vectorm,S|\varrho_\infty,\vectorm_\infty,S_\infty)}\text{ in } L^{\infty}_{\text{weak-(*)}}(0,T;\mathcal{M}(\R^d)). $$
    We introduce defect measures:
    \begin{itemize}
    	\item \textbf{Concentration defect:}\\
    	$ \mathfrak{R}^{\text{cd}}= \overline{\mathrm{e}(\varrho,\vectorm,S|\varrho_\infty,\vectorm_\infty,S_\infty)}- \langle \Nu_{t,x}; \mathrm{e}(\tilde{\varrho},\tilde{\vectorm},\tilde{S}|\varrho_\infty,\vectorm_\infty,S_\infty) \rangle ,$
    	\item \textbf{Oscillation defect:}\\
    	$ 
    	\mathfrak{R}^{\text{od}}=  \langle \Nu_{t,x}; \mathrm{e}(\tilde{\varrho},\tilde{\vectorm},\tilde{S}|\varrho_\infty,\vectorm_\infty,S_\infty) \rangle -\mathrm{e}(\varrho,\vectorm,S|\varrho_\infty,\vectorm_\infty,S_\infty), $
    	\item \textbf{Total  relative energy defect:}\\
    	$ \mathfrak{R}=  \mathfrak{R}^{\text{cd}}+\mathfrak{R}^{\text{od}}.$
    \end{itemize}
\begin{Rem}\label{rem-def-1}
	 As a direct consequence of \eqref{rel_eng} and Lemma \ref{LA1} we obtain 
	 \begin{align}\nonumber
	 \Vert \mathfrak{C}_{\varrho} \Vert_{L^{\infty}(0,T; \mathcal{M}(\R^d))} \leq  \Vert \mathfrak{R} \Vert_{L^{\infty}(0,T; \mathcal{M}(\R^d))} .
	 \end{align}
	 Similarly, we have
	 \begin{align*}
	 \Vert \vert \mathfrak{C}_\vectorm \vert \Vert_{L^{\infty}(0,T; \mathcal{M}(\R^d))}  + \Vert \vert \mathfrak{C}_S \vert \Vert_{L^{\infty}(0,T; \mathcal{M}(\R^d))} \leq \Vert \mathfrak{R} \Vert_{{L^{\infty}(0,T; \mathcal{M}(\R^d))} }.
	 \end{align*} 
\end{Rem}
\subsubsection{Finiteness of energy defect}
The definition of relative energy and \eqref{eng_bd_1} imply
\[ \Vert \mathrm{e}(\varrho_n,\vectorm_n,S_n)- \mathrm{e}(\varrho_\infty,\vectorm_{\infty},S_\infty) \Vert_{ L^{\infty}(0,T; L^2+ L^1(\R^d))} \leq C.\] 
In particular, we conclude 
\[ \begin{split}
\mathrm{e}(\varrho_n,\vectorm_n,S_n) -\mathrm{e}(\varrho_\infty,\vectorm_{\infty},S_\infty)  \rightarrow \overline{\mathrm{e}(\varrho,\vectorm,S)}-\mathrm{e}(\varrho_\infty,\vectorm_{\infty},S_\infty) \\
 \text{ weak(*)ly in } L^{\infty}_{\text{weak-(*)}}((0,T; L^2+\mathcal{M}(\R^d)) .
\end{split}
\]
Next we state a lemma that concludes the finiteness of the energy  defect
    \begin{Lemma}
    	Consider $ \mathfrak{R}_{\text{eng}} = \overline{\mathrm{e}(\varrho,\vectorm,S)}-{\mathrm{e}(\varrho,\vectorm,S)} $. Then $  \mathfrak{R}_{\text{eng}}  \in L^{\infty}(0,T;\mathcal{M}(
    	R^d))$ with $ \mathfrak{R}_{\text{eng}} (t)(\R^d)<\infty$ for a.e. $ t \in (0,T) $.
    \end{Lemma}
\begin{proof}
	We observe that \[ \begin{split}
	&\mathrm{e}(\varrho_{ n},\vectorm_{ n} S_n)- {\mathrm{e}(\varrho,\vectorm,S)} \\ 
	&= \mathrm{e}(\varrho_n,\vectorm_n,S_n|\varrho_\infty,\vectorm_\infty,S_\infty)-\mathrm{e}(\varrho,\vectorm,S|\varrho_\infty,\vectorm_\infty,S_\infty) \\
	&+ \partial \mathrm{e}(\varrho_\infty,\vectorm_\infty,S_\infty) \cdot (\varrho_{ n}-\varrho, \vectorm_{ n}- \vectorm, S_n-S)
	\end{split}\]
	From the above discussion along with Remark \ref{rem-def-1} we prove the result.
\end{proof}
\begin{Rem}
	In particular we have
	\[\mathfrak{R}= \mathfrak{R}_{\text{eng}}  -\partial \mathrm{e}(\varrho_\infty,\vectorm_\infty,S_\infty) \cdot (\mathfrak{C}_\varrho, \mathfrak{C}_\vectorm,\mathfrak{C}_S).\]
\end{Rem}
\begin{Rem}
	Convexity and lower semi-continuity of the map $ (\varrho,\vectorm,S)\mapsto \mathrm{e}((\varrho,\vectorm,S)) $ implies that \[\mathfrak{R}_{\text{eng}} \in L^{\infty}_{\text{weak-(*)}}((0,T;\mathcal{M}^{+}(
	R^d)) .\]
\end{Rem}

\subsubsection{Defect measures of the non-linear terms in momentum equation}
In approximate momentum equation we notice the presence of two non-linear terms $ \mathbb{1}_{\varrho_{ n}>0}\frac{\vectorm_{ n}\otimes \vectorm_{ n}}{\varrho_{ n}} $ and $ \mathbb{1}_{\varrho_{ n}>0}p(\varrho_{ n},S_n) $. We observe that 
\[ \left\Vert \mathbb{1}_{\varrho_{ n}>0}\frac{\vectorm_{ n}\otimes \vectorm_{ n}}{\varrho_{ n}} - \frac{\vectorm_\infty \otimes \vectorm_\infty}{\varrho_\infty} \right\Vert_{ L^{\infty}(0,T; L^2+ L^1(\R^d;\R^{d\times d}))} \leq C.  \] Thus we consider the \textbf{concentration defect} $ \mathfrak{C}_{m_1}^{\text{eng,cd}} $ and the \textbf{oscillation defect} $ \mathfrak{C}_{m_1}^{\text{eng,od}} $ as
\begin{align}\nonumber
\mathfrak{C}_{m_1}^{\text{eng,cd}}=  \overline{\frac{\vectorm\otimes \vectorm}{\varrho}}- \bigg \langle \Nu_{t,x}; \frac{\tilde{\vectorm}\otimes\tilde{ \vectorm}}{\tilde{\varrho}}  \bigg \rangle
\end{align}
and 
\begin{align}\nonumber
\begin{split}
\mathfrak{C}_{m_1}^{\text{eng,od}}= \bigg \langle \Nu_{t,x};\frac{\tilde{\vectorm}\otimes\tilde{ \vectorm}}{\tilde{\varrho}} \bigg \rangle- \mathbb{1}_{\varrho>0} \frac{\vectorm\otimes \vectorm}{\varrho}.
\end{split}
\end{align}
Similarly for the pressure term we define
\begin{align*}
\mathfrak{C}_{m_2}^{\text{eng,cd}}=  \overline{p(\varrho,S)\mathbb{I}}- \big \langle \Nu_{t,x}; p(\tilde{\varrho},\tilde{S})\mathbb{I} \big \rangle
\end{align*}
and 
\begin{align*}
\begin{split}
\mathfrak{C}_{m_2}^{\text{eng,od}}= \big \langle \Nu_{t,x}; p(\tilde{\varrho},\tilde{S})\mathbb{I} \big \rangle-p(\varrho,S)\mathbb{I}.
\end{split}
\end{align*}
We consider the total defect as $ \mathfrak{C}_{\text{eng}}=\mathfrak{C}_{m_1}^{\text{eng,cd}}+\mathfrak{C}_{m_1}^{\text{eng,od}}+\mathfrak{C}_{m_2}^{\text{eng,cd}}+\mathfrak{C}_{m_2}^{\text{eng,od}} $. \\
 For any $ \xi \in \R^d $, the function 
\begin{equation}\label{mdotxiconvex}
[\varrho,\vectorm ] \mapsto \begin{cases}
&\frac{\vert \vectorm \cdot \xi \vert^2 }{\varrho } \text{ if }\varrho>0,\\
&0,  \text{ if } \varrho = \vectorm  =0,\\
&\infty,  \text{ otherwise }
\end{cases}
\end{equation}
is convex lower semi-continuous. By virtue of \eqref{mdotxiconvex} we conclude 
\begin{align}\nonumber
\mathfrak{C}_{\text{eng}} \in  L^{\infty}_{\text{weak-(*)}}((0,T;\mathcal{M}^{+}(\R^d;\R^{d\times d}_{\text{sym}})).
\end{align}

\subsubsection{Comparison of defect measures $ \text{trace}(\mathfrak{C}_{\text{eng}}) $ and $ \mathfrak{R}_{\text{eng}} $}
With the help of the following relation
\[\text{trace}\bigg(\frac{\vectorm \otimes \vectorm}{\varrho}\bigg) = \frac{\vert \vectorm \vert^2}{\varrho} \text{ and } \text{trace}\bigg( \varrho^{\gamma} \exp{\bigg(\frac{S}{c_v \varrho}\bigg)}\mathbb{I}\bigg) = d  \varrho^{\gamma} \exp{\bigg(\frac{S}{c_v \varrho}\bigg)} \]
we conclude the existence of $ \Lambda_1,\Lambda_2>0 $ such that
\begin{align}\label{comp_def}
\Lambda_1 \mathfrak{R}_{\text{eng}}\leq \text{ trace}( \mathfrak{C}_{\text{eng}})\leq \Lambda_2 \mathfrak{R}_{\text{eng}}.
\end{align}

    \subsection{Limit passage and proof of the theorem \ref{th1}}
    Note that the main goal here is to pass the limit in continuity equations and momentum equation. 
    \subsubsection{Continuity equation} We perform the passage of limit in approximate continuity equation \eqref{app_cont} and obtain
   \begin{align*}
   \int_0^T \int_{\R^d} \big[\partial_{t} \phi  \text{ d}\overline{\varrho}(t) +   \nabla_{x} \phi \cdot \text{d}\overline{\vectorm} \big] \dt =0 ,
   \end{align*}
   for $ \phi \in C_c^1{((0,T)\times \R^d)} $. 
   In a more suitable notation we write
   \begin{align}\label{cont_eqn_w_def}
   \int_0^T \int_{\R^d} \big[\varrho \partial_{t} \phi +  \vectorm \cdot \nabla_{x} \phi \big] \dx \dt + \int_0^T \int_{ \R^d}\big[ \partial_{t}\phi \text{ d}\mathfrak{C}_{\varrho}+ \nabla_{x} \phi \cdot \text{d} \mathfrak{C}_{\vectorm}\big] \dt =0 ,
   \end{align}
   for $ \phi \in C_c^1{((0,T)\times \R^d)} $.
   	   Further we prove that \[ \overline{\varrho} \in C_{\text{weak(*)}}([0,T]; L^2+\mathcal{M}(\R^d)).\]
   	Using \eqref{app_ic1_1.2} we conclude 
   	\begin{align}\label{IC_meas}
   	\int_{K} \varrho_{ 0} \psi  \dx = \int_{ K} \psi\text{d}(\overline{\varrho}(0)),   
   	\end{align}
   for $  K\subset \R^d $, $ K  $ compact and $ \psi\in C_c(K) $.
   \subsubsection{Local equi-integrability of $ \{\varrho_{ n}\}_{n \in \mathbb{N}} $ and $ \{\vectorm_{ n}\}_{n\in \mathbb{N}} $}
   We assume the triplet $(\varrho, \vectorm, S)$ is a weak solution of complete Euler system with initial data $ (\varrho_{ 0},\vectorm_{0},S_0) $, i.e. equation of continuity reads as 
   \begin{align}\label{weak_cont_def}
   \int_0^T \int_{\R^d} \big[\varrho \partial_{t} \phi + \vectorm \cdot \nabla_{x} \phi \big] \dx \dt =-\int_{\R^d} \varrho_{ 0} \phi(0,\cdot)  \dx,
   \end{align}
   for any $ \phi \in C_c^1([0,T)\times \R^d) $.{
   	Eventually $ \varrho\in L^1_{\text{loc}}((0,T)\times\R^d) $ and $ \vectorm \in L^1_{\text{loc}}((0,T)\times\R^d;\R^d)  $ yield
   	\begin{align}\label{IC_young}
   	\int_{K} \varrho_{ 0} \psi  \dx = \int_{ K} \varrho(0,\cdot) \psi \dx ,   
   	\end{align}
   	for $ K $ compact subset of $ \R^d $ and $ \psi \in C_c(K) $. }\\
   On the other hand \eqref{weak_cont_def} along with \eqref{cont_eqn_w_def} imply
   \begin{align*}
   \partial_{t} \mathfrak{C}_{\varrho} + \Div \mathfrak{C}_{\vectorm} =0 
   \end{align*}
   in the sense of  distributions in $ (0,T)\times \R^d $.
   Using the fact $ \mathfrak{C}_\varrho \in  L^{\infty}_{\text{weak-(*)}}((0,T; \mathcal{M}(\R^d))$ and $ \mathfrak{C}_{\vectorm} \in  L^{\infty}_{\text{weak-(*)}}((0,T; \mathcal{M}(\R^d;\R^d)) $ we write the above relation as,
   \begin{align}\nonumber
   \begin{split}
   \int_0^T \int_{ \R^d} \partial_{t} \phi \text{ d}\mathfrak{C}_\varrho \dt +  \int_0^T \int_{ \R^d} \nabla_{x} \phi \cdot \text{ d}\mathfrak{C}_\vectorm \dt=0,  \text{ for } \phi  \in \mathcal{D}((0,T)\times  \R^d).
   \end{split}
   \end{align}
   We consider $ \phi(t,x)=\eta(t)\psi(x) $ with $ \eta \in \mathcal{D}(0,T) $ and $ \psi \in \mathcal{D}(\R^d) $.
   We rewrite the above  equation as 
   \begin{align}\nonumber
   \begin{split}
   \int_0^T \bigg( \int_{ \R^d}  \psi \text{ d}\mathfrak{C}_\varrho\bigg) \eta^{\prime}(t) \dt +  \int_0^T \bigg(  \int_{ \R^d} \nabla_{x} \psi \cdot \text{ d}\mathfrak{C}_\vectorm \bigg)\eta (t)\dt=0.
   \end{split}
   \end{align}
   Now using lemma  \eqref{Def_lem2} we observe that, for $ \eta \in \mathcal{D}(0,T) $ and $ \psi \in C^1(\R^d) $ we have 
   \begin{align}\nonumber
   \begin{split}
   \int_0^T \bigg( \int_{ \R^d}  \psi \text{ d}\mathfrak{C}_\varrho\bigg) \eta^{\prime}(t) \dt +  \int_0^T \bigg(  \int_{ \R^d} \nabla_{x} \psi \cdot \text{ d}\mathfrak{C}_\vectorm \bigg)\eta (t)\dt=0.
   \end{split}
   \end{align}
   Considering $ \psi=1 $ we obtain
   \begin{align}\nonumber
   \begin{split}
   \int_0^T \bigg( \int_{ \R^d}  \text{ d}\mathfrak{C}_\varrho\bigg) \eta^{\prime}(t) \dt =0.
   \end{split}
   \end{align}
   Here we can conclude that $ t \mapsto \int_{ \R^d}  \text{ d}\mathfrak{C}_\varrho(t) $ is absolutely continuous in $ (0,T) $ with is distributional derivative $ 0 $.\\
   \eqref{IC_meas} and \eqref{IC_young} lead us to conclude $ \mathfrak{C}_\varrho(0,\cdot)=0 $ in $ \R^d $. This implies,
   	\begin{align*}
   	\int_{ \R^d} \text{ d}\mathfrak{C}_\varrho(t) =0  \text{ for }t\in (0,T).
   	\end{align*}
   Hence we conclude $ \mathfrak{C}_\varrho =0$ for a.e. $ t\in (0,T) $. \\
   Let $ B \subset (0,T)\times \R^d $ {be a bounded Borel set.}
  Since $ \varrho_{ n}\geq 0 $ {and $ \mathfrak{C}_\varrho =0$} , 
	we conclude that $ \{\varrho_{ n}\}_{n\in \mathbb{N}} $ is equi-integrable in $ B $.\\
  
  We have $ \vectorm_{ n} = \sqrt{\varrho_{ n}} \frac{\vectorm_{ n}}{\sqrt{\varrho_{ n}}} $ and $ \frac{\vert\vectorm_{ n}\vert^2}{\varrho_{ n}}  $ is bounded in $ L^1(B) $. As a consequence of that we conclude 
  $ \{\vectorm_{ n}\}_{n\in \mathbb{N}} $ is {equi-integrable in $ B $}.
   \subsubsection{Momentum equation with defect }
   Now passing {to} limit in the momentum equation \eqref{app_mom}, it holds that
   \begin{align}\label{momentum_eqn_w_d_2}
   \begin{split}
   &\int_0^T \int_{\R^d} \bigg[ \vectorm \cdot \partial_{t} \vectorphi + \mathbb{1}_{\{\varrho>0\}}\frac{\vectorm \otimes \vectorm}{\varrho}: \nabla_{x} \vectorphi + \mathbb{1}_{\{\varrho>0\}} p(\varrho,S) \Div \vectorphi \bigg] \dx \dt \\
   &  + \int_0^T \int_{ \R^d} \nabla_{x} \vectorphi : \text{d}\mathfrak{C}_\text{eng} =0,
   \end{split}
   \end{align}
   for $ \vectorphi \in C_c((0,T)\times \R^d;\R^d) $.
   \subsubsection{Almost everywhere convergence}
    From our assumption that barycenter of the {Young} measure solves complete Euler system weakly implies 
   \begin{align*}
   \int_{ \R^d} \nabla_{x} \phi : \text{d}\mathfrak{C}_{\text{eng}}=0 \text{ for any }\phi\in C_c^1(\R^d;\R^d) \text{ for a.e. } t\in (0,T).
   \end{align*}
   Using Proposition \ref{Def_prop2} we conclude
   $  \mathfrak{C}_\text{eng} =0 $.
   Comparison of the defect measure implies
   $ \mathfrak{R}_\text{eng} =0 $. \\
   As a consequence of theorem \ref{wl1}  we have
   \begin{equation}\label{locweak}
   \mathrm{e}(\varrho_{ n},\vectorm_n, S_n) \rightarrow\mathrm{e}(\varrho,\vectorm,S)  \text{ weakly in } L^1{(B)} .
   \end{equation}
   Hence we deduce that
   \begin{align*}
   \overline{\mathrm{e}(\varrho,\vectorm,S)}=  \langle \Nu_{t,x}; \mathrm{e}(\tilde{\varrho},\tilde{\vectorm},\tilde{S}) \rangle = \mathrm{e}(\varrho,\vectorm,S) \text{ in } B.
   \end{align*} 
   Since $ \mathrm{e} $ is convex and strictly convex  in it's domain of positivity, we use a \emph{sharp form of Jensen's inequality} as described in Lemma 3.1, \cite{FH2019} to conclude that either,
   \begin{align*}
   \Nu_{t,x}=\delta_{\{\varrho(t,x), \vectorm(t,x), S(t,x)\}}
   \end{align*}
   or,
   \begin{align*}
   \text{supp}[\Nu]\subset \{ [\tilde{\varrho},\tilde{\vectorm},\tilde{S}] \vert \tilde{\varrho}=0,\; \tilde{\vectorm}=0,\; \tilde{S}\leq 0 \}.
   \end{align*}
   Recall our assumption \eqref{assump_ent_app1}, i.e.\[ S(t,x)=0 \text{ whenever } \varrho(t,x)=0  \text{ for  a.e. }(t,x) \in(0,T)\times \R^d . \] It implies $ \Nu_{t,x}=\delta_{\{\varrho(t,x), \vectorm(t,x), S(t,x)\}} $. From Lemma \ref{ym1} we conclude 
   $ \{ \varrho_{ n}, \vectorm_n, S_n \} $ converges to $ (\varrho,\vectorm,S) $ in measure. Passing to a suitable subsequence we obtain
   \begin{align}\label{a.e}
   \varrho_{ n} \rightarrow \varrho ,\; \vectorm_{ n} \rightarrow \vectorm \text{ and } S_n \rightarrow S \text{ a.e. in } (0,T) \times \R^d.
   \end{align}

\section{Convergence of approximate solutions from second approximation problem }
In this section our goal is to provide a proof of the theorem \eqref{th2}. From our formulation of the \textbf{second approximation problem } and hypothesis on the initial data \eqref{app_ic_2} {we can deduce}
 \emph{ the minimum principle for entropy}\eqref{ent_min_app2}, i.e., $ s_n\geq s_0 $. This helps us to achieve a finer estimate for the relative energy compared to \eqref{rel_eng}, that reads
	\begin{align}\label{rel_eng-2}
	\begin{split}
	\mathrm{e}(\varrho,\vectorm,S|\varrho_\infty,\vectorm_\infty,S_\infty)\geq  \begin{cases}
	(\varrho-\varrho_\infty)^2 + \vert \vectorm-\vectorm_\infty \vert^2 + (s-s_\infty)^2 \\
	\quad \text{ if } \frac{\varrho_\infty}{2} \leq \varrho \leq 2\varrho_\infty, \ \text{ and } \ |S| \leq 2\vert S_\infty \vert ,\\
	(1+ \varrho^\gamma) + \frac{\vectorm^2}{\varrho} + (1+S^\gamma),
	\\
	\quad \text{otherwise}
	\end{cases}
	\end{split}
	\end{align}
	One can {find} a detailed discussion about the above statement in Breit et al. \cite{BeFH2019}.
Without loss of generality we assume $ s_0 \geq 0 $, otherwise  we need to do a re-scaling by defining total entropy $ S_n {=} \varrho_{ n}(s_n-s_0) $.

\subsection{Uniform bounds}
We assume initial relative energy is uniformly bounded in \eqref{app_ic_1}. It implies
\begin{align*}
\Vert \mathrm{e}(\varrho_n,\vectorm_n,S_n|\varrho_\infty,\vectorm_\infty,S_\infty) \Vert_{L^\infty(0,T;L^1(\Om))} \leq C.
\end{align*}
This along with \eqref{rel_eng-2} gives us
\begin{align}\label{den,mom_2}
\begin{split}
&\Vert \varrho_n - \varrho_\infty \Vert_{L^{\infty}(0,T;L^\gamma+L^2(\R^d))} \leq C,\\
&\big\Vert \vectorm_n - \vectorm_\infty \big\Vert_{L^{\infty}(0,T;L^{\frac{2\gamma}{\gamma+1}}+L^2(\R^d))}\leq C.
\end{split}
\end{align}
Finally recalling the total entropy $S_n$ we have
\begin{align}\label{entr_2}
\begin{split}
&\Vert S_n - S_\infty \Vert_{L^{\infty}(0,T;L^\gamma+L^2(\R^d))} \leq C,\\
&\left\Vert  \frac{S_n}{\sqrt{\varrho_{ n}}} \right\Vert_{L^\infty(0,T;L^{2\gamma}(\Om))} \leq C.
\end{split}
\end{align}
\subsection{Weak Convergence}
From \eqref{den,mom_2} and \eqref{entr_2} we observe 
\begin{align*}
&\varrho_{ n}-\varrho_\infty \rightarrow \varrho-\varrho_\infty \text{ weak(*)ly in } L^{\infty}(0,T;L^{\gamma}+L^2(\R^d)),\\
&\vectorm_{n}-\vectorm_\infty \rightarrow \vectorm - \vectorm_\infty \text{ weak(*)ly in }  L^{\infty}(0,T;L^{\frac{2\gamma}{\gamma+1}}+L^2(\R^d)),\\
&S_{ n}-S_\infty \rightarrow S-S_\infty \text{ weak(*)ly in } L^{\infty}(0,T;L^{\gamma}+L^2(\R^d)),
\end{align*} 
passing to suitable subsequences as the case may be. Here also one can consider a Young measure $ \mathcal{V}  $ generated by $ (\varrho_{ n},\vectorm_{ n},S_n) $ such that \begin{align}
\mathcal{V} \in L^{\infty}_{\text{weak-(*)}}((0,T)\times \Om;\mathcal{P}(\R^{d+2} )).
\end{align} 
Since Young measures capture the weak limit we obtain 
\begin{align}\nonumber
\begin{split}
&(\varrho(t,x),\vectorm(t,x),S(t,x))\\
&=(\{(t,x)\mapsto \langle \mathcal{V}_{t,x}; \tilde{\varrho}\rangle\},  \{(t,x)\mapsto \langle \mathcal{V}_{t,x}; \tilde{\vectorm}\rangle\}, \{(t,x)\mapsto \langle \mathcal{V}_{t,x}; \tilde{S}\rangle\}).
\end{split}
\end{align}
\subsection{Defect measure}  Unlike Section 5, here we have the presence of a defect measure only in non linear terms. 
\subsubsection{Relative energy defect}
We know \[ L^{\infty}(0,T;L^1(\R^d)) \subset L^{\infty}_{\text{weak-(*)}}(0,T;\mathcal{M}(\R^d)). \]
In addition, $L^{\infty}_{\text{weak-(*)}}(0,T;\mathcal{M}(\R^d))$ is the dual of $ L^1(0,T; C_0(\R^d)) $, hence passing to a suitable subsequence we obtain,
 $$ \mathrm{e}(\varrho_n,\vectorm_n,S_n|\varrho_\infty,\vectorm_\infty,S_\infty)\rightarrow \overline{\mathrm{e}(\varrho,\vectorm,S|\varrho_\infty,\vectorm_\infty,S_\infty)}\text{ in } L^{\infty}_{\text{weak-(*)}}(0,T;\mathcal{M}(\R^d)). $$
 In particular we say
\[
\mathrm{e}_{\text{kin}}(\varrho_n,\vectorm_n|\varrho_\infty,\vectorm_{\infty})\rightarrow \overline{\mathrm{e}_{\text{kin}}(\varrho,\vectorm|\varrho_\infty,\vectorm_{\infty})} \text{ in } L^{\infty}_{\text{weak-(*)}}(0,T;\mathcal{M}(\R^d))  \]
and
\[\mathrm{e}_{\text{int}}(\varrho_n,S_n|\varrho_\infty,S_{\infty})\rightarrow \overline{\mathrm{e}_{\text{int}}(\varrho,S|\varrho_\infty,S_{\infty})} \text{ in } L^{\infty}_{\text{weak-(*)}}(0,T;\mathcal{M}(\R^d))  .\]
Using convexity and lower semi-continuity we have,
\begin{align}
\begin{split}
\mathfrak{R}_{e}=& \bigg( \overline{\mathrm{e}_{\text{kin}}(\varrho,\vectorm|\varrho_\infty,\vectorm_{\infty})}-\bigg(\frac{1}{2}\mathbb{1}_{\{\varrho>0\}} \frac{\vert \vectorm \vert^2}{\varrho}-\vectorm \cdot \vectoru_\infty+ \frac{1}{2}\varrho \vert \vectoru_\infty\vert^2 \bigg) \bigg)\\
&\quad+ \bigg( \overline{\mathrm{e}_{\text{int}}(\varrho,S|\varrho_\infty,S_{\infty})} - \bigg(\mathbb{1}_{\{\varrho>0\}}\mathrm{e}_{\text{int}}(\varrho,S) -\frac{\partial \mathrm{e}_{\text{int}}}{\partial \varrho}(\varrho_\infty, S_\infty) (\varrho-\varrho_\infty)\\
&\quad-\frac{\partial \mathrm{e}_{\text{int}}}{\partial S}(\varrho_\infty, S_\infty) (S-S_\infty) -\mathrm{e}_{\text{int}} (\varrho_\infty,S_\infty)\bigg)\bigg) \in L^{\infty}_{\text{weak-(*)}}(0,T;\mathcal{M}^{+}(\R^d)).
\end{split}
\end{align}
\subsubsection{Defects from the non linear terms in momentum equation }
We consider a map $ \mathbb{C} (\cdot, \cdot| \varrho_\infty,\vectoru_\infty):\R \times \R^d \rightarrow \R^{d\times d}  $ as $$ \mathbb{C}(\varrho,\vectoru| \varrho_{\infty},\vectoru_\infty)= \varrho (\vectoru-\vectoru_\infty)\otimes (\vectoru - \vectoru_\infty) .$$ Let, $ \xi \in \R^d $,  then with the help of $  \vectorm =\varrho \vectoru $ and $ \vectorm_\infty=\varrho_\infty \vectoru_{\infty} $, we conclude that the map 
\begin{align*}
(\varrho,\vectorm) \mapsto \mathbb{C}(\varrho,\vectorm| \varrho_\infty,\vectorm_\infty): (\xi \otimes \xi)
\end{align*}
is a convex function. We  have
\[ \frac{\vectorm_{n}\otimes \vectorm_n}{\varrho_{ n}}= \mathbb{C}(\varrho_n,\vectorm_n| \varrho_{\infty},\vectoru_\infty)+ \vectorm_n \otimes \vectoru_\infty + \vectoru_\infty\otimes \vectorm_{n} - \varrho_{ n} \vectoru_\infty\otimes \vectoru_\infty, \]
with
$$ \Vert \mathbb{C}(\varrho_{ n},\vectoru_n\vert \varrho_{\infty},\vectoru_\infty) \Vert_{L^{\infty}(0,T;L^1(\R^d;\R^{d\times d}))} \leq C .$$
It implies
\begin{align*}
\mathbb{C}(\varrho_{ n},\vectorm_n\vert \varrho_{\infty},\vectoru_\infty) \rightarrow \overline{\mathbb{C}(\varrho,\vectorm\vert \varrho_\infty,\vectoru_\infty)} \text{ weakly in } L^{\infty}_{\text{weak-(*)}}(0,T;\mathcal{M}(\R^d;\R^{d\times d}_{\text{sym}})).
\end{align*}
Define,
\begin{align}
\mathfrak{R}_{m_1}=  \overline{\mathbb{C}(\varrho,\vectorm\vert \varrho_\infty,\vectoru_\infty)} -\bigg[\frac{1}{2}\frac{ \vectorm \otimes \vectorm }{\varrho}- \vectorm \otimes \vectoru_\infty - \vectoru_\infty\otimes \vectorm + \varrho \vectoru_\infty\otimes \vectoru_\infty \bigg]
\end{align}
Similarly we define a map $ \mathbb{P}(\cdot, \cdot| \varrho_\infty,S_\infty) :\R \times \R \rightarrow \R^{d\times d}   $ such that
\begin{align*}
&\mathbb{P}(\varrho,S\vert \varrho_{\infty},S_\infty)\\&= \bigg(p(\varrho,S)-\frac{\partial p}{\partial \varrho}(\varrho_\infty, S_\infty) (\varrho-\varrho_\infty)-\frac{\partial p}{\partial S}(\varrho_\infty, S_\infty) (S-S_\infty) -p (\varrho_\infty,S_\infty)\bigg)\mathbb{I}.
\end{align*}
We define the defect measure
\begin{align}
\mathfrak{R}_{m_2}=  \overline{\mathbb{P}(\varrho,S\vert \varrho_\infty,S_\infty)} -\mathbb{P}(\varrho,S\vert \varrho_\infty,S_\infty ).
\end{align} 
Using \eqref{mdotxiconvex} we conclude,
\begin{align}
\mathfrak{R}_{m}= \mathfrak{R}_{m_1}+\mathfrak{R}_{m_2} \in  L^{\infty}_{\text{weak-(*)}}(0,T;\mathcal{M}^{+}(\R^d;\R^{d\times d}_{\text{sym}}) 
\end{align}
\subsubsection{Comparison of defect measures}
There exists scalars $ \Lambda_1,\Lambda_2 >0 $ such that
\begin{align}\label{comp_def_2}
\Lambda_1 \mathfrak{R}_{e} \leq \text{trace}(\mathfrak{R}_{m})\leq \Lambda_2\mathfrak{R}_{e}.
\end{align}
\begin{Rem}
	Clearly one can notice that here we do not need to define energy defect separately like section 5. Basically weak convergence of the state variables imply that the energy defect coincides with relative energy defect.  
\end{Rem}
\subsection{Passage of limit}
Now we pass limit in the equations of approximate solutions and obtain
\paragraph{Equation of continuity:}
\begin{align}
\int_0^T \int_{\R^d} \big[\varrho \partial_{t} \phi +  \vectorm \cdot \nabla_{x} \phi \big] \dx \dt=0 ,
\end{align}
for any $ \phi \in C_c^1((0,T)\times \R^d) $.
\paragraph{Momentum equation with defect:}
\begin{align}
\begin{split}
&\int_0^T \int_{\R^d} \bigg[ \vectorm \cdot \partial_{t} \vectorphi + \mathbb{1}_{\{\varrho>0\}}\frac{\vectorm \otimes \vectorm}{\varrho}: \nabla_{x} \vectorphi +\mathbb{1}_{\{\varrho>0\}}\ p(\varrho,S) \Div \vectorphi \bigg] \dx \dt \\
&+ \int_0^T \int_{ \R^d} \nabla_{x} \vectorphi : \text{d}\mathfrak{R}_{m}=0 ,
\end{split}
\end{align}
for any $ \vectorphi \in C_c^1((0,T)\times \R^d;\R^d) $.
\paragraph{Relative energy:}
\begin{align}
\begin{split}
\overline{\mathrm{e}(\varrho,\vectorm,S|\varrho_\infty,\vectorm_\infty,S_\infty)}= \mathrm{e}(\varrho,\vectorm,S|\varrho_\infty,\vectorm_\infty,S_\infty)+ \mathfrak{R}_{e}.
\end{split}
\end{align}
\subsection{Proof of the Theorem \ref{th2}}
\subsubsection{Disappearance of defect measures} 
We assume the triplet $(\varrho, \vectorm, S)$ is a weak solution of complete Euler system. It implies
\begin{align*}
\int_0^T \int_{ \R^d} \nabla_{x} \vectorphi : \text{d}\mathfrak{R}_{m}=0,
\end{align*}
for any $ \vectorphi \in C_c^1([0,T]\times \R^d;\R^d) $.
Thus applying Proposition \eqref{Def_prop2} we conclude $ \mathfrak{R}_{m}=0 $. Finally using \eqref{comp_def_2} we obtain $ \mathfrak{R}_e=0 $.\par 
Thus we also have,
\begin{align}
\begin{split}
\mathrm{e}(\varrho_{ n},\vectorm_{n},S_n|\varrho_\infty,\vectorm_\infty,S_\infty)\rightarrow \mathrm{e}(\varrho,\vectorm,S|\varrho_\infty,\vectorm_\infty,S_\infty)\\
\text{weak(*)ly  in }L^{\infty}_{\text{weak-(*)}}(0,T;\mathcal{M}(\R^d)).
\end{split}
\end{align}

\subsubsection{Almost everywhere convergence}
Let $ B \subset (0,T)\times \R^d$ be a compact set. Recall the Young measure generated by $ \{(\varrho_{ n},\vectorm_{ n},S_n)\}_{n\in \mathbb{N}} $ is $ \Nu $. From $ \mathfrak{R}_{e} =0$ we infer that
\[ \big\langle \Nu_{t,x};  \mathrm{e}(\tilde{\varrho},\tilde{\vectorm},\tilde{S}|\varrho_\infty,\vectorm_\infty,S_\infty) \big\rangle = \mathrm{e}(\varrho,\vectorm,S|\varrho_\infty,\vectorm_\infty,S_\infty)\text{ for a.e. }(0,T)\times \R^d.\]
We already have weak(*) convergence of $ \{\mathrm{e}(\varrho_{ n},\vectorm_{n},S_n|\varrho_\infty,\vectorm_\infty,S_\infty)\}_{n\in \mathbb{N}} $, using Lemma \eqref{wl1} we deduce that
\begin{align}
\begin{split}
\mathrm{e}(\varrho_{ n},\vectorm_{n},S_n|\varrho_\infty,\vectorm_\infty,S_\infty)\rightarrow \mathrm{e}(\varrho,\vectorm,S|\varrho_\infty,\vectorm_\infty,S_\infty)
\text{ weakly in }L^{1}(B).
\end{split}
\end{align}
Now convexity of $ \mathrm{e}(\cdot\vert \varrho_\infty, \vectorm_\infty,S_\infty) $ and Theorem 2.11 from Feireisl \cite{F2004b} helps us to conclude
\begin{align}
\varrho_{ n} \rightarrow \varrho, \vectorm_{ n} \rightarrow \vectorm  \text{ and }S_n\rightarrow S \text{ a.e. in }B.
\end{align}

\subsubsection{Local strong convergence}
We have $ \{\mathrm{e}(\varrho_n,\vectorm_n,S_n|\varrho_\infty,\vectorm_\infty,S_\infty)\}_{n\in \mathbb{N}} $ is equi-integrable in $ B $, in particular $ \{\mathrm{e}_{\text{int}}(\varrho_{ n},S_n)\}_{n\in \mathbb{N}} $ is equi-integrable in $ B $. As a trivial consequence we obtain \\
$ \{(\varrho_{ n}^\gamma,S_n^\gamma)\}_{n\in \mathbb{N}} $ is also equi-integrable.
Above statement along with almost everywhere convergence gives
\begin{align*}
\varrho_{ n}^\gamma \rightarrow \varrho^{\gamma} \text{ and }S_n^\gamma\rightarrow S^\gamma \text{weakly in } L^{1}(B).
\end{align*}
It implies
\begin{align}
\int_{ B} \varrho_{ n}^\gamma \dx \dt \rightarrow \int_{ B} \varrho^\gamma \dx \dt
\text{ and }\int_{ B} S_{ n}^\gamma \dx \dt \rightarrow \int_{ B} S^\gamma \dx \dt.\end{align}
These concludes the norm convergence i.e., \begin{align*}
\vert \varrho_{ n} \vert_{L^\gamma(B)} \rightarrow \vert \varrho \vert_{L^\gamma(B)}.
\end{align*}
Now weak convergence and norm convergence implies the strong convergence.
\begin{align*}
\varrho_{ n} \rightarrow \varrho\; \text{ in } L^{\gamma}(B).
\end{align*}
Similarly, for the total entropy we also obtain,
\begin{align*}
S_n \rightarrow S \text{ in } L^{\gamma}(B).
\end{align*}
To prove strong convergence of $ \{\vectorm_n\}_{n\in \mathbb{N}} $, first we consider $\mathbf{h}_n\equiv \frac{\vectorm_n}{\sqrt{\varrho_{ n}}}$.
We have
\begin{align*}
	\mathbf{h}_n\rightarrow \mathbf{h} \text{ weakly in } L^2(B;\R^d),\text{ for some }\mathbf{h} \in L^2(B).
\end{align*}
We already have
\[ \vectorm_n \rightarrow \vectorm \text{ weakly in } L^1(B;\R^d). \]
Strong convergence of $ \varrho_{ n} $ implies
\begin{align*}
	\sqrt{\varrho_{ n}} \mathbf{h}_n =\mathbf{m}_n \rightarrow \vectorm = \sqrt{\varrho} \mathbf{h} \text{ weakly in }L^1(B;\R^d).
\end{align*}
We observe that the set $ \{(t,x) \in B \vert \varrho(t,x)=0,\;  \vectorm(t,x)\neq 0\}$ is of zero Lebesgue measure. 
Thus we conclude,
\[ \frac{\vectorm_n}{\sqrt{\varrho_{ n}}} \rightarrow \mathbb{1}_{\{\varrho>0\}}\frac{\vectorm}{\sqrt{\varrho}} \text{ weakly in } L^2(B;\R^d) \]
Using the fact $$ \int_{ {B}} \frac{\vert \vectorm_{ n} \vert^2}{\varrho_{ n}} \dx \dt \rightarrow  \int_{ {B}} \mathbb{1}_{\{\varrho>0\}}\frac{\vert \vectorm \vert^2}{\varrho} \dx \dt $$ and strong convergence of $ \{\varrho_{ n}\}_{\{n\in \mathbb{N}\}} $ implies 
\begin{align}
	\vectorm_{ n} \rightarrow \vectorm \text{ in } L^{1}(B;\R^d).
\end{align} Since $ \varrho \in L^{\gamma}(B) $ we deduce that
\begin{align*}
\vectorm_{ n} \rightarrow \vectorm \text{ in } L^{\frac{2\gamma}{\gamma+1}}({B};\R^d).
\end{align*}
Relative energy is positive, lower semi-continuous and convex function. It implies
\[ \mathrm{e}(\varrho_{ n},\vectorm_{n},S_n|\varrho_\infty,\vectorm_\infty,S_\infty)\rightarrow \mathrm{e}(\varrho,\vectorm,S|\varrho_\infty,\vectorm_\infty,S_\infty) \text{ in } L^{1}(B). \]
We invoke the bounds \eqref{den,mom_2} and \eqref{entr_2} to conclude our desired strong convergences as stated in theorem. 

\centerline{ \bf Acknowledgments}
\vspace{2mm}

The work is supported by Einstein Stiftung, Berlin. I would like to thank Prof. E. Feireisl for his valuable suggestions and comments.

\end{document}